\documentclass[12pt]{amsart}
\usepackage{amsmath}
\usepackage{amstext}
\usepackage{amsfonts}
\usepackage{amssymb}
\usepackage{amsthm}
\usepackage{amsrefs}
\usepackage{bm}

\usepackage{microtype}
\usepackage{color,hyperref}
\definecolor{darkblue}{rgb}{0.0,0.0,0.3}
\hypersetup{colorlinks,breaklinks,
            linkcolor=red,urlcolor=red,
            anchorcolor=red,citecolor=red}

\usepackage{tikz-cd}
\usetikzlibrary{arrows}
\usetikzlibrary{matrix,arrows,decorations.pathmorphing}

\theoremstyle{plain}
\newtheorem{thm}{Theorem}[section]

\newtheorem{cor}[thm]{Corollary}
\newtheorem{prop}[thm]{Proposition}
\newtheorem{lem}[thm]{Lemma}

\theoremstyle{definition}
\newtheorem{defn}[thm]{Definition}
\newtheorem{example}[thm]{Example}

\numberwithin{equation}{section}

\newcommand{\B}{{\mathcal{B}}}
\newcommand{\bC}{{\mathbb{C}}}

\newcommand{\bN}{{\mathbb{N}}}

\newcommand{\bR}{{\mathbb{R}}}

\newcommand{\V}{{\mathcal{V}}}

\begin{document}
\title[An infinite quantum ramsey theorem]{An infinite quantum ramsey theorem}

\author[M. Kennedy]{Matthew Kennedy}
\address{Department of Pure Mathematics\\ University of Waterloo\\
Waterloo\\Ontario\\N2L 3G1\\Canada}
\email{matt.kennedy@uwaterloo.ca}

\author[T. Kolomatski]{Taras Kolomatski}
\address{Department of Pure Mathematics and Mathematical Statistics\\ University of Cambridge\\
Cambridge\\ CB3 0WA\\United Kingdom}
\email{tk452@cam.ac.uk}

\author[D. Spivak]{Daniel Spivak}
\address{Department of Mathematics\\ University of Toronto\\
Toronto\\Ontario\\M5S 2E4\\Canada}
\email{daniel.spivak@mail.utoronto.ca}

\begin{abstract}
We prove an infinite Ramsey theorem for noncommutative graphs realized as unital self-adjoint subspaces of linear operators acting on an infinite dimensional Hilbert space. Specifically, we prove that if $\V$ is such a subspace, then provided there is no obvious obstruction, there is an infinite rank projection $P$ with the property that the compression $P \V P$ is either maximal or minimal in a certain natural sense.
\end{abstract}

\thanks{First author supported by NSERC Grant Number 418585.}
\thanks{}

\maketitle

\section{Introduction}

The notion of a noncommutative graph first appeared in the work of Duan, Severini and Winter \cite{DSW2013} in quantum information theory as a generalization, for a quantum channel, of the confusability graph of a channel from classical information theory. In this setting, a noncommutative graph takes the form of an operator system, which is a unital self-adjoint subspace of operators acting on a Hilbert space. 

The idea of viewing an operator system as a noncommutative generalization of a graph appears to be much more broadly applicable. For example, the work of Stahlke \cite{S2016} and Weaver \cite{W2015}, as well as many others (see e.g. \cites{PO2015, LPT2017, KM2017}) shows that there are useful generalizations for operator systems of many graph-theoretic constructions and results. 

Recently, Weaver \cite{W2017} proved an analogue of Ramsey's theorem for operator systems acting on finite dimensional spaces. This result has concrete applications, for example, to the theory of quantum error correction. But perhaps more interesting, Weaver's work suggests that operator systems provide a new and potentially very interesting setting for Ramsey theory.

In this paper, we adopt this perspective and prove an analogue of the infinite Ramsey theorem for operator systems acting on infinite dimensional spaces.

For a complex Hilbert space $H$, let $\B(H)$ denote the space of bounded (i.e. continuous) linear operators on $H$. If $H$ is finite dimensional, say $\dim H = n$, then $\B(H)$ can be identified with the $n \times n$ matrices $M_n(\bC)$.

A subspace $\V \subseteq \B(H)$ is said to be an {\em operator system} if it is unital, meaning that it contains the identity operator $I_H$, and self-adjoint, meaning that $A^* \in \V$ whenever $A \in \V$, where $A^*$ denotes the adjoint of $A$. For a general reference on operator systems, we direct the reader to Paulsen's book \cite{P2002}.

An easy way to see an operator system as a generalization as a graph is to consider the operator system associated to a graph $G = (V,E)$ with vertex set $V$ and edge set $E$. Let $H_G$ be a Hilbert space with $\dim H_G = |V|$ and let $(h_v)_{v \in V}$ be an orthonormal basis for $H_G$ indexed by $V$. For vectors $x,y \in H_G$, define a rank one operator $xy^* \in \B(H_G)$ by $xy^*(z) = \langle z, y \rangle x$ for $z \in H_G$. The operator system $\V_G \subseteq \B(H_G)$ corresponding to $G$ is defined by
\[
\V_G = \operatorname{span} \{I_{H_G}, h_v h_w^* : (v,w) \in E \}.
\]

It is clear that $\V_G$ completely encodes $G$. A result of Paulsen and Ortiz \cite{PO2015} implies that if $G' = (V',E')$ is another graph with $|V| = |V'|$, then $G$ and $G'$ are isomorphic if and only if the corresponding operator systems $\V_G$ and $\V_{G'}$ are completely order isomorphic, i.e. isomorphic in the category of operator systems.

Recall that a subset of vertices $C \subseteq V$ is said to be an {\em anticlique} if there are no edges between any pair of vertices in $C$. On the other hand, $C$ is said to be a {\em clique} if every pair of distinct vertices in $C$ is adjacent.

Let $P_C \in \B(H_G)$ denote the projection onto $\{ h_v : v \in C \}$. It is easy to see that $C$ is an anticlique if and only if $P_C \V_G P_C = \bC P_C$, where $P_C \V_G P_C$ denotes the compression $P_C \V_G P_C = \{P_C A P_C : A \in \V_G\}$. If $C$ has finite cardinality, say $|C| = k$, then it is also easy to see that $C$ is a clique if and only if the compression $P_C \V_G P_C$ is maximal, in the sense that $\dim P_C \V_G P_C = k^2$ or equivalently, $P_C \V_G P_C = P_C \B(H_G) P_C$.

Let $\V \subseteq \B(H)$ be an arbitrary operator system. Motivated by the above considerations, for $k \geq 1$ Weaver says that a rank $k$ projection $P \in \B(H)$ is a {\em quantum $k$-anticlique} for $\V$ if $P \V P = \bC P$. On the other hand, Weaver says that $P$ is a {\em quantum $k$-clique} for $\V$ if $P \V P = P \B(H) P$.

Ramsey's theorem \cite{R1930} asserts that for $k \geq 1$ there is $N \geq 1$ such that every graph with at least $N$ vertices has either a $k$-clique or a $k$-anticlique. Weaver's noncommutative Ramsey theorem \cite{W2017} is an analogue of this result for operator systems.

\begin{thm}[Weaver]
Let $\V \subseteq \B(H)$ be an operator system on a finite dimensional Hilbert space $H$. For $k \geq 1$ there is $N \geq 1$ such that if $\dim H \geq N$, then $\V$ has either a quantum $k$-clique or a quantum $k$-anticlique.
\end{thm} 

In order to motivate the definition of an infinite quantum clique, consider an infinite graph $G$ with corresponding operator system $\V_G$ as above. Let $C \subseteq V$ be a subset of vertices and let $P_C$ denote the projection onto $\{ h_v : v \in C \}$ as above. We equip $\B(H_G)$ with the weak (or weak operator) topology, which is the topology of coordinate-wise convergence, meaning that a net of operators $(T_i)$ in $\B(H_G)$ weakly converges to an operator $T \in \B(H_G)$ if and only if $\lim_i \langle T_i x, y \rangle = \langle T x, y \rangle$ for all $x,y \in H_G$. It is easy to check that $C$ is an infinite clique if and only if $P_C$ has infinite rank and $P_C \V_G P_C$ is weakly dense in $P_C \B(H_G) P_C$.

\begin{defn} \label{defn:main}
Let $\V \subseteq \B(H)$ be an operator system on an infinite dimensional Hilbert space $H$. An {\em infinite quantum clique} of $\V$ is an infinite rank orthogonal projection $P \in \B(H)$ with the property that $P\V P = \{P V P : V \in \V\}$ is weakly dense in $P\B(H)P$, i.e. the weak closure of the compression $P\V P$ is maximal. An {\em infinite quantum anticlique} of $\V$ is an infinite rank orthogonal projection $P \in \B(H)$ with the property that $P\V P = \bC P$, i.e. the compression $P \V P$ is minimal. An {\em infinite quantum obstruction} is an infinite rank projection $P \in \B(H)$ with the property that $P \V P$ is finite dimensional, every operator in $P \V P$ can be written as the sum of a scalar multiple of $P$ and a compact operator and $P \V P$ contains a self-adjoint compact operator with range dense in $PH$.
\end{defn}

The infinite Ramsey theorem asserts that a graph on infinitely many vertices has either an infinite anticlique or an infinite clique. The main result in our paper is an analogue of this result for operator systems.

\begin{thm} \label{thm:main}
Let $\V \subseteq \B(H)$ be an operator system on an infinite dimensional Hilbert space $H$. Then it has either an infinite quantum clique, an infinite quantum anticlique or an infinite quantum obstruction.
\end{thm}

As a corollary to Theorem \ref{thm:main}, we obtain a dichotomy that is slightly simpler to state but provides less information. Before stating it, we require the following definition.

\begin{defn} \label{defn:almost-anticlique}
Let $\V \subseteq \B(H)$ be an operator system on an infinite dimensional Hilbert space $H$. An {\em infinite quantum almost anticlique} of $\V$ is an infinite rank projection $P \in \B(H)$ such that the compression $P \V P$ is finite dimensional and every operator in $P \V P$ can be written as the sum of a scalar multiple of $P$ and a compact operator.
\end{defn}

\begin{cor} \label{cor:main}
Let $\V \subseteq \B(H)$ be an operator system on an infinite dimensional Hilbert space $H$. Then it has either an infinite quantum clique or an infinite quantum almost anticlique.
\end{cor}

The next two examples show that there are operator systems on infinite dimensional Hilbert spaces with no infinite quantum anticlique and no infinite quantum clique. By Theorem \ref{thm:main} an operator system with this property necessarily has an infinite quantum obstruction.

\begin{example}
Let $\V \subseteq \B(\ell^2(\bN))$ be any finite dimensional operator system containing a self-adjoint compact operator $K$ with dense range. Since $\V$ is finite dimensional, it clearly has no infinite quantum clique. On the other hand, for any infinite rank orthogonal projection $P \in \B(\ell^2(\bN))$, the compression $PKP$ is nonzero and compact, so $P \V P \ne \bC P$. Therefore, $\V$ also has no infinite quantum anticlique. The identity operator is an infinite quantum obstruction for $\V$.
\end{example}

The next example is an infinite dimensional variant of the example from \cite{W2017}*{Proposition 2.3}. 

\begin{example} \label{ex:weaver}
Let $H = \ell^2(\bN)$ and let $(e_n)$ denote the standard orthonormal basis of $H$. Let $K \in \B(H)$ denote the compact operator defined by
\[
K = \sum_{n \geq 1} \frac{1}{n} e_n e_n^*,
\]
where the sum is taken in the weak topology. Note that $K$ has dense range. Define an operator system $\V \subseteq \B(H)$ by
\[
\V = \operatorname{span}\{I_H, K, e_1 e_n^*, e_n e_1^* : n \geq 1 \}.
\]

Let $P \in \B(H)$ be any infinite rank projection. Then $P \V P$ contains the nonzero compact operator $PKP$, so $P$ is not an infinite quantum anticlique for $\V$. We will show that $P$ is not an infinite quantum clique for $\V$.

First suppose $Pe_1 = 0$. Then for $n \geq 1$, $P e_1 e_n^* P = 0$ and $Pe_n e_1^* P = 0$, so $P \V P = \operatorname{span} \{ P, PKP  \}$. Hence $P$ is an infinite quantum obstruction for $\V$.

Next suppose $Pe_1 \ne 0$. Let $(f_n)$ be an orthonormal basis for $PH$ with $f_1 = Pe_1/\|Pe_1\|$. For $n \geq 1$, let $v_n = Pe_n$. Then $P e_1 e_n^* P = \|Pe_1\| f_1 v_n^*$ and $P e_n e_1^* P = \|Pe_1\| v_n f_1^* $, so
\[
P \V P = \operatorname{span}\{P, PKP, f_1 v_n^*, v_n f_1^* : n \geq 1\}.
\]
To see that $P$ is not an infinite quantum clique for $\V$, note that if $P \V P$ is weakly dense in $P \B(H) P$, then for any infinite rank projection $Q \in \B(H)$ with $Q \leq P$, $Q \V Q$ is weakly dense in $Q \B(H) Q$. In other words, if $P$ is an infinite quantum anticlique for $\V$, then so is $Q$. However, let $Q$ denote the infinite rank projection onto the closure of $\operatorname{span}\{f_n : n \geq 2\}$. Then $Q \leq P$, but for $n \geq 1$, $Q f_1 v_n^* Q = 0$ and $Q v_n f_1^* Q = 0$, so $Q \V Q = \operatorname{span}\{Q, QKQ\}$. Hence $Q$ is an infinite quantum obstruction for $\V$.
\end{example}

In the definition of an infinite quantum clique, it is tempting to replace the condition that $P \V P$ is weakly dense in $P \B(H) P$ with the condition that $P \overline{\V} P = P\B(H)P$, where $\overline{\V}$ denotes the weak closure of $\V$. The next example shows that our main result would not hold with this definition.

\begin{example}
Let $H = \ell^2(\bN \cup \{0\})$ and let $(e_n)$ denote the standard orthonormal basis of $H$. For $n \geq 1$, let
\[
A_n = e_n e_n^* +  n e_0 e_n^* + n e_n e_0^*,
\]
and define an operator system $\V_0 \subseteq \B(H)$ by $\V_0 = \operatorname{span}\{I, A_n : n \geq 1\}$. Then every operator $A \in \V_0$ satisfies $\langle A e_i, e_j \rangle = 0$ for $i,j \geq 1$ with $i \ne j$.

Let $\V$ denote the weak closure of $\V_0$. Then $\V$ is a weakly closed operator system and every operator $B \in \V$ satisfies $\langle B e_i, e_j \rangle = 0$ for $i,j \geq 1$ with $i \ne j$. In particular, every operator in $\V$ can be written as the sum of a diagonal operator and an operator of rank at most two.

For $m \geq 1$, let
\[
T_m =  2m e_0 e_0^* + e_0 e_m^* + e_m e_0^* - 2m e_m e_m^*.
\]
Then $\operatorname{tr}(T_m I) = 0$ and $\operatorname{tr}(T_m A_n) = 0$ for all $n \geq 1$, so $\operatorname{tr}(T_m A) = 0$ for all $A \in \V_0$. Hence $\operatorname{tr}(T_m B) = 0$ for all $B \in \V$, or equivalently,
\[
2m \langle B e_m, e_m \rangle = \langle B e_0, e_m \rangle + \langle B e_m, e_0 \rangle + 2m \langle B e_0, e_0 \rangle.
\]

For $B \in \V$ with $\|B\| \leq 1$, let $B' = B - \langle B e_0, e_0 \rangle I$. Then $\|B'\| \leq 2$ and $\langle B' e_0, e_0 \rangle = 0$, so it follows from above that for $m \geq 1$,
\[
|\langle B' e_m, e_m \rangle| = \frac{1}{2m} |\langle B' e_0, e_n \rangle + \langle B' e_n, e_0 \rangle| \leq \frac{2}{m}.
\]
In particular, $\lim_m \langle B' e_m, e_m \rangle = 0$. Hence $B'$ can be written as the sum of a compact diagonal operator and an operator of rank at most two. In particular, $B'$ is compact. It follows that $\V \subseteq \bC I_H + \mathcal{K}(H)$, where $\mathcal{K}(H)$ denotes the space of compact operators on $H$. Since $\mathcal{K}(H)$ is an ideal in $\B(H)$, this implies that if $P \in \B(H)$ is a projection, then $P \V P \subseteq \bC P + \mathcal{K}(H)$. However, if $P$ has infinite rank, then $\bC P + \mathcal{K}(H)$ is a proper subset of $P \B(H) P$. Hence there is no infinite rank projection $P \in \B(H)$ such that $P \V P = P \B(H) P$.

Next, we claim that $\V$ does not have an infinite quantum anticlique. To see this, suppose for the sake of contradiction that $P \in \B(H)$ is an infinite rank projection such that $P \V P = \bC P$. For $n \geq 1$, the fact that $P A_n P$ is compact implies that $P A_n P = 0$. Hence for nonzero $f \in PH \cap \{e_0\}^\perp$,
\begin{align*}
0 &= \langle A_n f, f \rangle \\
&= \langle f, e_n \rangle \langle e_n, f \rangle + n \langle f, e_n \rangle \langle e_0, f \rangle + n \langle e_0, f \rangle \langle e_n, f \rangle \\
&= |\langle f, e_n \rangle|^2.
\end{align*}
Since $\{e_n : n \geq 1\}$ is an orthonormal basis for $\{e_0\}^\perp$, this implies $f = 0$. Hence $PH \cap \{e_0\}^\perp = 0$. But it is not difficult to show that $PH \cap \{e_0\}^\perp \ne 0$ (see Lemma \ref{lem:intersection}), giving a contradiction.

Finally, we claim that $\V$ does not have an infinite quantum obstruction. To see this, suppose for the sake of contradiction that $P \in \B(H)$ is an infinite rank projection with the property that $P \V P$ is finite dimensional and contains a self-adjoint compact operator $K$ with range dense in $PH$. Since $P \V P$ is finite dimensional, so is $P \V_0 P$, say
\[
P \V_0 P = \operatorname{span}\{P,PA_{n_1}P,\ldots,PA_{n_k}P\}.
\]
But $P \V P$ is contained in the weak closure of $P \V_0 P$ and, being finite dimensional, $P \V_0 P$ is already weakly closed. So $P \V P = P \V_0 P$. In particular, $P \V P$ is spanned by $P$ and finitely many finite rank operators, so it does not contain any compact operators of infinite rank, giving a contradiction.
\end{example}

The notion of a quantum anticlique arises in quantum information theory, specifically in the theory of quantum error correction. Let $H$ and $K$ be Hilbert spaces and let $\mathcal{E} : \B_t(H) \to \B_t(K)$ be a quantum channel, i.e. a completely positive trace preserving map between the spaces $\B_t(H)$ and $\B_t(K)$ of trace class operators on $H$ and $K$ respectively (see e.g. \cite{BKK2009} for details). Then there are operators $(E_i)_{i \in I}$ in $\B(H,K)$ such that
\[
\mathcal{E}(T) = \sum_{i \in I} E_i T E_i^*
\]
for $T \in \B_t(H)$, where the sum is taken in the weak topology. The operator system associated to $\mathcal{E}$ is
\[
\V_\mathcal{E} = \operatorname{span}\{I_H, E_i^* E_j : i,j \in I\}.
\]
In the finite dimensional setting, a quantum anticlique for $\V_\mathcal{E}$ corresponds to an error correcting code for $\mathcal{E}$ \cites{KL1997,KLV2000}. In the infinite dimensional setting, an infinite quantum anticlique also gives rise to an error correcting code for $\mathcal{E}$ \cite{BKK2009}.

In Section \ref{sec:dilation}, we establish a key technical result about dilations of operator systems. To prove our main result, we first reduce to the case of a diagonalizable operator system in Section \ref{sec:reduction}. The case of a diagonalizable operator system is then handled in Section \ref{sec:diagonalizable}. The main result is proved in Section \ref{sec:main}.

\subsection*{Acknowledgements}

The authors are grateful to Nik Weaver for suggesting the inclusion of Corollary \ref{cor:main} and Example \ref{ex:weaver}.

\section{Dilation lemma} \label{sec:dilation}

\begin{lem} \label{lem:technical}
Let $\V \subseteq \B(H)$ be an operator system on an infinite dimensional Hilbert space $H$. Let $K$ be another infinite dimensional Hilbert space and let $T \in \B(H,K)$ be an operator such that $T \V T^*$ is weakly dense in $\B(K)$. Then the projection $P \in \B(H)$ onto $\ker(T)^\perp$ is an infinite quantum clique for $\V$.
\end{lem}

\begin{proof}
Let $T^* = V^*S$ be the polar decomposition for $T^*$, so that $V^* \in \B(K,H)$ is an isometry and $S = |T^*| = (TT^*)^{1/2} \in \B(K)$. Then $P = V^*V$. Let $\V'$ denote the operator system  $\V' = V \V V^* \subseteq \B(K)$. Then we can write
\[
T \V T^* = S V \V V^* S = S \V' S.
\]

By the spectral theorem, there is a projection-valued measure $E$ on $\bR$ such that
\[
S = \int_0^{\|S\|} s\,dE_s
\]
For $t > 0$, define operators $S_t, R_t \in \B(K)$ by
\[
R_t = \int_t^{\|S\|} s^{-1}\,dE_s, \qquad S_t = \int_t^{\|S\|} s\,dE_s.
\]
Also, let $E_t = E((t,\|S\|))$. Then $R_t S = R_t S_t = E_t$.

For $t > 0$,
\[
E_t \V' E_t = R_t S \V' S R_t = R_t T \V T^* R_t.
\]
Since $T \V T^*$ is weakly dense in $\B(K)$, $E_t T \V T^* E_t$ is weakly dense in $E_t\B(K)E_t$. Since $R_t$ is invertible in $E_t \B(K) E_t$, this implies that $R_t T \V T^* R_t$ (and hence $E_t \V' E_t$) is weakly dense in $E_t \B(K) E_t$.

Since $(E_t)$ weakly converges to $I_K$, it follows that $\V'$ is weakly dense in $\B(K)$. Hence the compression $P \V P = V^* \V' V$ is dense in $P \B(H) P$.
\end{proof}

\begin{lem} \label{lem:dilation}
Let $\V \subseteq \B(H)$ be an operator system on an infinite dimensional Hilbert space $H$. Suppose there is an isometry $V : \ell^2(\bN) \to H \oplus H$ such that $V^* (\V \oplus 0_H) V$ is weakly dense in $\B(\ell^2(\bN))$. Then $\V$ has an infinite quantum clique.
\end{lem}

\begin{proof}
Define $T \in \B(H,\ell^2(\bN))$ by $Th = V^*(h,0)$ for $h \in H$. Then $T \V T^* = V^* (\V \oplus 0_H) V$ is weakly dense in $\B(\ell^2(\bN))$ by assumption. We conclude from Lemma \ref{lem:technical} that $\V$ has an infinite quantum clique.
\end{proof}

\begin{lem} \label{lem:isometry}
Let $(A_n)$ be a sequence of operators on a countably infinite dimensional Hilbert space $H$ with orthonormal basis $(h_n)$. For $n \geq 1$ and $i,j \geq 1$, let $\alpha_{ij}^n = \langle A_n h_j, h_i \rangle$, so that
\[
A_n = \sum_{i,j \geq 1} \alpha_{ij}^n h_i h_j^*,
\]
where the sum is taken in the weak topology. Let $(x_n)$ be a sequence of vectors in $\ell^2(\bN)$ such that $\|x_n\| \leq 2^{-(n+1)}$ for $n \geq 1$. For $n \geq 1$, let
\[
T_n = \sum_{i,j \geq 1} \alpha_{ij}^n x_i x_j^*,
\] 
where the sum is taken in the weak topology. Then there is an isometry $V : \ell^2(\bN) \to H \oplus H$ such that for $n \geq 1$,
\[
V^* (A_n \oplus 0_H) V = T_n.
\]
\end{lem}

\begin{proof}
Let $(e_n)$ denote the standard orthonormal basis for $\ell^2(\bN)$. Define a linear operator $R : \ell^2(\bN) \to \ell^2(\bN)$ by $R e_n = x_n$ for $n \geq 1$. Since $\|x_n\| \leq 2^{-(n+1)}$, it is not difficult to check that $\|R\| \leq 1$. Define a linear operator $S : \ell^2(\bN) \to \ell^2(\bN) \oplus \ell^2(\bN)$ by $S = R \oplus (1-R^*R)^{1/2}$. Then $S$ is an isometry, so the sequence $(y_n)$ in $\ell^2(\bN) \oplus \ell^2(\bN)$ defined by $y_n = S e_n$ for $n \geq 1$ is orthonormal. Let $Y$ denote the closure of $\operatorname{span}\{y_n : n \geq 1\}$.

Define an isometry $W : Y \to H \oplus H$ by $W y_n = h_n \oplus 0_H$ for $n \geq 1$. Then we can extend $W$ to an isometry from $\ell^2(\bN)$ to $H \oplus H$ that we continue to denote by $W$.

Let $Q_1 : \ell^2(\bN) \oplus \ell^2(\bN) \to \ell^2(\bN)$ denote the projection onto the first summand and define $V : \ell^2(\bN) \to H \oplus H$ by $V = WQ_1^*$. Then $V$ is an isometry and $V^* (h_n \oplus 0_H) = Q_1 W^* W y_n = Q_1 y_n = x_n$ for $n \geq 1$. Hence
\[
V^* (A_n \oplus 0_H) V = \sum_{i,j \geq 1} \alpha_{ij}^k V^* (h_i \oplus 0_H) (h_j \oplus 0_H)^* V = \sum_{i,j \geq 1} \alpha_{ij}^n x_i x_j^* = T_n.
\]
\end{proof}

\section{Reduction to the diagonalizable case} \label{sec:reduction}

For a Hilbert space $H$ and a (not necessarily closed) subspace $X \subseteq H$, we will write $H \ominus X$ for $H \cap X^\perp$.

\begin{lem} \label{lem:infinite-codim-or-diagonalizable}
Let $\V \subseteq \B(H)$ be an operator system on an infinite dimensional Hilbert space $H$. Then there is an infinite rank projection $P \in \B(H)$ satisfying one of the following conditions:
\begin{enumerate}
\item For all nonzero vectors $x \in PH$, $\dim PH \ominus P\V Px < \infty$.
\item The compression $P \V P$ is diagonalizable.
\end{enumerate}
In particular, if $\dim \V < \infty$, then there is an infinite rank projection $P \in \B(H)$ such that the compression $P \V P$ is diagonalizable.
\end{lem}
\begin{proof}
Suppose there is no infinite rank projection $P \in \B(H)$ satisfying the first condition. Let $P_1 = I_H$. By assumption there is $x_1 \in H$ such that $\dim H \ominus \V P_1 x_1 = \infty$. We inductively define a decreasing sequence of infinite rank projections $(P_n)$ in $\B(H)$ and unit vectors $(x_n)$ in $H$ such that for $n \geq 2$, $P_n H = P_{n-1} H \ominus P_{n-1} \V P_{n-1} x_{n-1}$ and $x_n \in P_n H$.

For $k \geq 2$, suppose $P_1,\ldots,P_{k-1}$ and $x_1,\ldots,x_{k-1}$ have been chosen.  Let $P_k \in \B(H)$ denote the projection onto $P_{k-1} H \ominus P_{k-1} \V P_{k-1} x_{k-1}$. Then by assumption $\dim P_k H = \infty$ and there is $x_k \in P_k H$ such that $\dim P_k H \ominus P_k \V P_k x_k = \infty$.

Proceeding inductively, we obtain the sequences $(P_n)$ and $(x_n)$. We claim $\langle A x_k, x_l \rangle = 0$ for $A \in \V$ and $k,l \geq 1$ with $k \neq l$. To see this, first note that since $\V$ is self-adjoint, we can assume $k < l$. Then by construction, $P_k \geq P_l$ so $x_l = P_l x_l = P_k x_l$ and $x_l \in P_k H \ominus P_k \V P_k x_k$. Hence
\[
\langle A x_k, x_l \rangle = \langle A P_k x_k, P_k x_l \rangle = \langle P_k A P_k x_k, x_l \rangle = 0.
\]
It follows that if $P$ is the projection onto the closure of $\operatorname{span}\{x_n : n \geq 1\}$, then the compression $P \V P$ is diagonalizable.
\end{proof}

\begin{lem} \label{lem:intersection}
Let $H$ be an infinite dimensional Hilbert space. Let $X \subseteq H$ be a (potentially non-closed) subspace and let $Y \subseteq H$ be a closed subspace. Suppose $\dim X > \dim Y^\perp$ and $\dim Y^\perp < \infty$. Then $X \cap Y \ne 0$.
\end{lem}

\begin{proof}
Let $n = \dim Y^\perp$ and let $x_1,\ldots,x_{n+1} \in X$ be linearly independent. For each $k$, write $x_i = y_i + z_i$ for $y_k \in Y$ and $z_k \in Y^\perp$. Since $\dim Y^\perp = n$, the vectors $z_1,\ldots,z_{n+1}$ are linearly dependent, and hence there are scalars $\alpha_1,\ldots,\alpha_{n+1} \in \bC$, not all zero, such that $\alpha_1 z_1 + \cdots + \alpha_{n+1} z_{n+1} = 0$. Since $x_1,\ldots,x_{n+1}$ are linearly independent, $0 \ne \alpha_1 x_1 + \cdots + \alpha_{n+1} x_{n+1} \in Y$.
\end{proof}

\begin{lem} \label{lem:inf-dim-range}
Let $\V \subseteq \B(H)$ be an operator system on an infinite dimensional Hilbert space $H$ such that for all nonzero vectors $h \in H$, $\dim \V h = \infty$. Then there is a sequence $(A_n)$ in $\V$ and an orthonormal sequence $(x_n)$ in $H$ such that for each $n$, $\langle A_n x_n, x_{n+1} \rangle = 1$ and $\langle A_k x_i, x_j \rangle = 0$ if $\max\{i,j\} > n+1$ and $i \ne j$.
\end{lem}

\begin{proof}
We inductively construct the sequences $(A_n)$ in $\V$ and $(x_n)$ in $H$ such that for $n \geq 1$, $A_n x_n = x_{n+1}$ and
\[
x_{n+1} \in \V x_n \cap \{x_i,\ A_j x_i,\ A_j^* x_i : 1 \leq i \leq n \text{ and } 1 \leq j \leq n-1 \}^\perp.
\]

Let $x_1 \in H$ be any unit vector. Then by assumption $\dim \V x_1 = \infty$, so by Lemma \ref{lem:intersection} there is a unit vector $x_2 \in \V x_1 \cap \{x_1\}^\perp$. Choose $A_1 \in \V$ such that $A_1 x_1 = x_2$.

For $k \geq 2$, suppose $A_1,\ldots,A_{k-1}$ and $x_1,\ldots,x_k$ have been chosen. By assumption $\dim \V x_k = \infty$, so by Lemma \ref{lem:intersection} there is a unit vector
\[
x_{k+1} \in \V x_k \cap \{x_i,\ A_j x_i,\ A_j^* x_i : 1 \leq i \leq k \text{ and } 1 \leq j \leq k-1 \}^\perp.
\]
Choose $A_k \in \V$ such that $A_k x_k = x_{k+1}$.

Proceeding inductively, we obtain the sequences $(A_n)$ and $(x_n)$. For $n \geq 1$, $\langle A_n x_n, x_{n+1} \rangle = \langle x_{n+1}, x_{n+1} \rangle = 1$. Finally, suppose $k,l \geq1$ with $\max\{k,l\} > n+1$ and $k \ne l$. If $k < l$ then $x_l \in \{A_n x_k\}^\perp$. Otherwise, if $k > l$, then $x_k \in \{A_n^* x_l\}^\perp$. Either way, $\langle A_n x_k, x_l \rangle = 0$.
\end{proof}

\begin{lem} \label{lem:exists-nonzero-fin-rank}
Let $\V \subseteq \B(H)$ be an operator system on a countably infinite dimensional Hilbert space. Let $(A_n)$ be a sequence in $\V$ and let $(h_n)$ be an orthonormal basis for $H$ such that for $n \geq 1$, $\langle A_n h_n, h_{n+1} \rangle = 1$ and $\langle A_n h_i, h_j \rangle  = 0$ if $\max\{i,j\} > n+1$ and $i \ne j$. Suppose that for all $N \geq 1$, there is a nonzero finite rank operator in $\operatorname{span}\{A_n : n \geq N \}$. Then there is a sequence $(B_n)$ of self-adjoint operators in $\V$ and an orthonormal sequence $(f_n)$ in $H$ such that for $n \geq 1$, $\langle B_n f_n, f_n \rangle = 1$ and $\langle B_n f_i, f_j \rangle = 0$ if $\max\{i,j\} > n$.
\end{lem}

\begin{proof}
We inductively construct the sequences $(B_n)$ in $\V$ and $(f_n)$ in $H$ along with a strictly increasing sequence $(N_n)$ in $\bN$ such that $B_n \in \operatorname{span}\{A_j,A_j^* : j > N_n \}$, $f_n \in \operatorname{span}\{h_i : N_{n-1} < i < N_n \}$, $\langle B_n f_n, f_n \rangle = 1$ and $\langle B_n f_i, f_j \rangle = 0$ if $\max\{i,j\} > n$.

For $N \geq 1$, let $P_N$ denote the projection onto $\operatorname{span}\{h_n : 1 \leq n \leq N\}$. By assumption there is a nonzero finite rank operator $F_1 \in \operatorname{span}\{A_j : j \geq 1 \}$. Choose $N_1 \geq 1$ such that $F_1 = P_{N_1} F_1 P_{N_1}$. If $\operatorname{Re}(F_1) \ne 0$, then let $B_1 = \operatorname{Re}(F_1)$. Otherwise, let $B_1 = \operatorname{Im}(F_1)$. Then $0 \ne B_1 = P_{N_1} B_1 P_{N_1}$. Hence there is unit vector $f_1 \in H$ such that $\langle B_1 f_1, f_1 \rangle \ne 0$. We can multiply $B_1$ by a scalar if necessary to ensure that $\langle B_1 f_1, f_1 \rangle = 1$.

For $k \geq 2$, suppose $B_1,\ldots,B_{k-1}$, $f_1,\ldots,f_{k-1}$ and $N_1,\ldots,N_{k-1}$ have been chosen. By assumption there is a nonzero finite rank operator $F_k \in \operatorname{span}\{A_j : j > N_{k-1}\}$. If $\operatorname{Re}(F_k) \ne 0$ then let $B_k = \operatorname{Re}(F_k)$. Otherwise, let $B_k = \operatorname{Im}(F_k)$. Choose $N_k \in \bN$ such that $B_k = P_{N_k} B_k P_{N_k}$.

By construction, $(P_{N_k} - P_{N_{k-1}}) B_k (P_{N_k} - P_{N_{k-1}}) \ne 0$. Hence there is a unit vector $f_k \in \operatorname{span} \{h_i : N_{k-1} < i \leq N_k \}$ such that $\langle B_k f_k, f_k \rangle \ne 0$. We can multiply $B_k$ by a scalar if necessary to ensure that $\langle B_k f_k, f_k \rangle = 1$.

Proceeding inductively, we obtain the sequences $(B_n)$ and $(f_n)$. For $n \geq 1$, $\langle B_n f_n, f_n \rangle = 1$. Furthermore, for $k,l \geq 1$ with $\max\{k,l\} \geq n$, $f_k,f_l \perp P_{N_n}H$. Since $B_n = P_{N_n} B_n P_{N_n}$, it follows that $\langle B_n f_k, f_l \rangle = 0$.
\end{proof}

The proof of the next lemma is inspired by the proof of Lemma \cite{W2017}*{Lemma 4.3}. 

\begin{lem} \label{lem:fin-supp}
Let $\V \subseteq \B(H)$ be an operator system on a countably infinite dimensional Hilbert space. Let $(A_n)$ be a self-adjoint sequence of operators in $\V$ and let $(h_n)$ be an orthonormal basis for $H$ such that for $n \geq 1$, $\langle A_n h_n, h_n \rangle = 1$ and $\langle A_n h_i, h_j \rangle = 0$ if $\max\{i,j\} > n$. Then $\V$ has an infinite quantum clique.
\end{lem}

\begin{proof}
Let $(e_n)$ be the standard orthonormal basis for $\ell^2(\bN)$. For $m \geq 1$ we can identify the $m \times m$ matrices $M_m(\bC)$ with the set of operators $T \in \B(\ell^2(\bN))$ satisfying $\langle T e_i, e_j \rangle = 0$ if $\max\{i,j\} > m$. For $n \geq 1$ and $i,j \geq 1$, let $\alpha_{ij}^n = \langle A_n h_j, h_i \rangle$, so that
\[
A_n = \sum_{1 \leq i,j \leq n} \alpha_{ij}^n h_i h_j^*,
\]
where the sum is taken in the weak topology. Note that since $A_n$ is self-adjoint, $\alpha_{ij}^n = \overline{\alpha_{ji}^n}$.

We will inductively construct sequences $(x_n)$ in $\ell^2(\bN)$ and $(T_n)$ in $\B(\ell^2(\bN))$. For $n \geq 1$, the operator $T_n$ will be defined by
\[
T_n = \sum_{1 \leq i,j \leq n} \alpha_{ij}^n x_i x_j^*.
\]
We will choose $x_n$ such that $\|x_n\| \leq 2^{-(n+1)}$ and $x_n \in \operatorname{span}\{e_i : 1 \leq i \leq m\}$ if $(m - 1)^2 < n \leq m^2$. In addition, we will require that  the operators $T_1,\ldots,T_n$ are linearly independent. This will imply that for $m \geq 1$, $\operatorname{span}\{T_1,\ldots,T_{m^2} \} = M_m(\bC)$.

Let $x_1 = 2^{-2} e_1$. For $k \geq 2$, suppose $x_1,\ldots,x_{k-1}$ have been chosen and let $m \geq 1$ satisfy $(m - 1)^2 < k \leq m^2$. Suppose for the moment that $x_k \in \ell^2(\bN)$ has been chosen. Let
\[
y = \sum_{1 \leq i \leq {k-1}} \alpha_{ik}^k x_i
\]
and let
\[
T = \sum_{1 \leq i,j \leq k-1} \alpha_{ij}^k x_i x_j^* - yy^*.
\]
Then since $\alpha_{kk}^k = 1$,
\begin{align*}
T_k &= \sum_{1 \leq i,j \leq k-1} \alpha_{ij}^k x_i x_j^* - yy^* + (y + x_k)(y + x_k)^* \\
&= T + yy^* + yx_k^* + x_k y^* + x_k x_k^*.
\end{align*}

Since $M_m(\bC)$ is spanned by positive rank one matrices and $k \leq m^2$, there is $z \in \operatorname{span}\{e_i : 1 \leq i \leq m\}$ such that $zz^* \notin \operatorname{span}\{T_1,\ldots,T_{k-1} \}$. It is easy to check there are at most finitely many $\alpha > 0$ such that
\[
T + yy^* + \alpha yz^* + \alpha zy^* + \alpha^2 z z^* \in \operatorname{span}\{T_1,\ldots,T_{k-1}\}.
\]
Hence there is a choice of $\alpha > 0$ such that for $x_k = \alpha z$, the operators $T_1,\ldots,T_k$ are linearly independent and $\|x_k\| \leq 2^{-(k+1)}$. Proceeding inductively, we obtain the sequences $(T_n)$ and $(x_n)$ in $\ell^2(\bN)$ as required.

We now apply Lemma \ref{lem:isometry} to obtain an isometry $V : \ell^2(\bN) \to H \oplus H$ such that $V^* (A_n \oplus 0_H) V = T_n$ for $n \geq 1$. By construction, $\operatorname{span}\{T_n : n \geq 1\}$ is weakly dense in $\B(\ell^2(\bN))$. Hence by Lemma \ref{lem:dilation}, we conclude that $\V$ has an infinite quantum clique.
\end{proof}

\begin{lem} \label{lem:no-fin-rank}
Let $\V \subseteq \B(H)$ be an operator system on a countably infinite dimensional Hilbert space. Let $(A_n)$ be a linearly independent sequence of operators in $\V$ and let $(h_n)$ be an orthonormal basis for $H$ such that for $n \geq 1$ there is $N_n \geq 1$ such that $\langle A_n h_i, h_j \rangle = 0$ if $\max\{i,j\} > N_n$ and $i \ne j$. Suppose there are no nonzero finite rank operators in $\operatorname{span}\{A_n : n \geq 1\}$. Then there is an infinite rank projection $P \in \B(H)$ and an orthonormal basis $(f_n)$ for $PH$ such that the sequence of compressions $(P A_n P)$ is linearly independent and for $n \geq 1$, $\langle A_n f_i, f_j \rangle = 0$ if $\max\{i^3,j^3\} > n$ and $i \ne j$.
\end{lem}

\begin{proof}
We can assume that the sequence $(N_n)$ is strictly increasing. We inductively construct a sequence $(f_n)$ in $H$ and a sequence $(M_n)$ of indices such that for $n \geq 1$, $f_n = h_{M_n}$, $M_n > N_{n^3}$ and the compressions $P_n A_1 P_n,\ldots,P_n A_n P_n$ are linearly independent, where $P_n$ denotes the projection onto $\operatorname{span}\{f_1,\ldots,f_n\}$.

Since $A_1$ has infinite rank there is $M_1 > N_1$ such that $\langle A_1 h_{M_1}, h_{M_1} \rangle \ne 0$. Let $f_1 = h_{M_1}$. For $k \geq 2$ suppose $f_1,\ldots,f_{k-1}$ and $M_1,\ldots,M_{k-1}$ have been chosen.

Suppose first there is an operator $A \in \operatorname{span}\{A_1,\ldots,A_{k-1}\}$ such that $P_{k-1} A P_{k-1} = P_{k-1} A_k P_{k-1}$. Since $P_{k-1} A_1 P_{k-1},\ldots,P_{k-1} A_{k-1} P_{k-1}$ are linearly independent by the induction hypothesis, $A$ is unique. Also, since $A_1,\ldots,A_k$ are linearly independent, $A - A_k$ has infinite rank, so there is $M_k > N_{k^3}$ such that $\langle (A - A_k) h_{M_k}, h_{M_k} \rangle \ne 0$. Then $\langle A h_{M_k}, h_{M_k} \rangle \ne \langle A_k h_{M_k}, h_{M_k} \rangle$. Otherwise, if there is no such operator $A \in \operatorname{span}\{A_1,\ldots,A_{k-1}\}$, then choose any $M_k > N_{k^3}$. Either way, letting $f_k = h_{M_k}$ and letting $P_k$ denote the projection onto $\operatorname{span}\{f_1,\ldots,f_k\}$, it follows that $P_k A_1 P_k, \ldots, P_k A_k P_k$ are linearly independent. Proceeding inductively, we obtain the sequence $(h_n)$. 

For $n \geq 1$ and $i,j \geq 1$ with $\max\{i^3,j^3\} > n$ and $i \ne j$, it follows from above that either $M_i > N_{i^3} > N_n$ or $M_j > N_{j^3} > N_n$. Hence
\[
\langle A_n f_i, f_j \rangle = \langle A_n h_{M_i}, h_{M_j} \rangle = 0.
\]
We can take $P$ to be the projection onto the closure of $\operatorname{span}\{f_n : n \geq 1\}$.
\end{proof}

\begin{lem} \label{lem:diag}
Let $\V \subseteq \B(H)$ be an operator system on a countably infinite dimensional Hilbert space. Let $(A_n)$ be a linearly independent sequence of operators in $\V$ and let $(h_n)$ be an orthonormal basis for $H$ such that for $n \geq 1$, $\langle A_n h_i, h_j \rangle = 0$ if $\max\{i^3,j^3\} > n$ and $i \ne j$. Then there is a sequence $(B_n)$ in $\V$ that is linearly independent and simultaneously diagonalizable.
\end{lem}

\begin{proof}
For $n \geq 1$ we can identify the $n \times n$ matrices $M_n(\bC)$ with the set of operators $T \in \B(H)$ satisfying $\langle T h_i, h_j \rangle = 0$ if $\max\{i,j\} > n$. For $n \geq 2$, the $3n^2-3n+1$ operators $A_{(n-1)^3+1},\ldots,A_{n^3}$ can each be written as the sum of an operator in $M_n(\bC)$ and an operator in $\B(H)$ that is diagonal with respect to the orthonormal basis $(h_n)$. Since $3n^2 - 3n + 1 \geq n^2 = \dim M_n(\bC)$, there is a nonzero operator $B_{n-1} \in \operatorname{span}\{A_{(n-1)^3+1},\ldots,A_{n^3}\}$ that is diagonal with respect to $(h_n)$. The linear independence of the sequence $(A_n)$ implies the linear independence of the sequence $(B_n)$.
\end{proof}

\begin{prop} \label{prop:pre-diagonalizable}
Let $\V \subseteq \B(H)$ be an operator system on an infinite dimensional Hilbert space. Suppose there is no infinite rank projection $P \in \B(H)$ such that the compression $P \V P$ is diagonalizable. Then either $\V$ has an infinite quantum clique or there is an infinite rank projection $P \in \B(H)$ and an operator system $\V' \subseteq \V$ such that the compression $P \V' P$ is infinite dimensional and diagonalizable.
\end{prop}

\begin{proof}
Suppose there is no infinite rank projection $P \in \B(H)$ such that the compression $P \V P$ is diagonalizable. Then by Lemma \ref{lem:infinite-codim-or-diagonalizable} there is an infinite rank projection $P \in \B(H)$ such that for every nonzero vector $x \in PH$, $\dim PH \ominus P \V P x < \infty$. By replacing $\V$ with $P \V P$, we can assume that for every nonzero vector $x \in H$, $\dim H \ominus \V x < \infty$. In particular, $\dim \V x = \infty$.

By Lemma \ref{lem:inf-dim-range}, there is a sequence $(A_n)$ in $\V$ and an orthonormal sequence $(x_n)$ in $H$ such that for $n \geq 1$, $\langle A_n x_n, x_{n+1} \rangle = 1$ and $\langle A_n x_i, x_j \rangle = 0$ if $\max\{i,j\} > n+1$ and $i \ne j$.

Suppose first that for all $N \geq 1$ there is a nonzero finite rank operator in $\operatorname{span}\{A_n : n \geq N\}$. Then applying Lemma \ref{lem:exists-nonzero-fin-rank} followed by Lemma \ref{lem:fin-supp}, we conclude that $\V$ has an infinite quantum clique.

Otherwise, by truncating $(A_n)$, we can assume that $\operatorname{span}\{A_n : n \geq 1\}$ does not contain any nonzero finite rank operators. Then we can apply Lemma \ref{lem:no-fin-rank}, followed by Lemma \ref{lem:diag} to obtain an infinite rank projection $P \in \B(H)$ and a sequence $(C_n)$ in $\V$ such that the sequence $(P C_n P)$ is linearly independent and diagonalizable. In this case we can take $\V' = \operatorname{span}\{C_n : n \geq 1\}$.
\end{proof}

\section{Diagonalizable case} \label{sec:diagonalizable}

\begin{lem} \label{lem:compression-compact}
Let $H$ be a countably infinite dimensional Hilbert space and let $A_1,\ldots,A_n \in \B(H)$ be self-adjoint and simultaneously diagonalizable. Then there is an infinite rank projection $P \in \B(H)$ and scalars $\alpha_1,\ldots,\alpha_n \in \bR$ such that for $k \geq 1$ the compression $P(A_k - \alpha_k I_H)P$ is self-adjoint and compact.
\end{lem}

\begin{proof}
First consider the case of a single self-adjoint diagonalizable operator $A \in \B(H)$. Let $(h_n)$ be an orthonormal basis for $H$ that diagonalizes $A$. For $n \geq 1$, let $\alpha_n = \langle A h_n, h_n \rangle$. Since $A$ is bounded, the sequence $(\alpha_n)$ is bounded. Hence there is a subsequence $(\alpha_{n_k})$ converging to $\alpha \in \bR$. Let $P$ denote the projection onto the closure of $\operatorname{span}\{h_{n_k} : k \geq 1\}$. Then the compression $P(A - \alpha I_H)P$ is a diagonalizable operator with diagonal entries converging to zero. In particular it is compact. The case of multiple operators follows from an iteration of the above argument.
\end{proof}

\begin{lem} \label{lem:diag-finite-dim-case}
Let $\V \subseteq \B(H)$ be a finite dimensional diagonalizable operator system on an infinite dimensional Hilbert space. Then either $\V$ has an infinite quantum anticlique or it has an infinite quantum obstruction.
\end{lem}

\begin{proof}
Choose self-adjoint operators $A_1,\ldots,A_n \in \V$ such that $\V = \operatorname{span}\{I_H,A_1,\ldots,A_n\}$. By Lemma \ref{lem:compression-compact}, there is an infinite rank projection $P \in \B(H)$ and scalars $\alpha_1,\ldots,\alpha_n \in \bR$ such that for $k \geq 1$, $P(A_k - \alpha_k I_H)P$ is self-adjoint and compact. Hence by replacing each $A_k$ with $P(A_k - \alpha_k I_h)P$ and replacing $\V$ with the compression $P \V P$, we can assume that each $A_k$ is compact.

Let $(h_n)$ be an orthonormal basis for $H$ that simultaneously diagonalizes $A_1,\ldots,A_n$. If each $A_k$ has finite rank, then for sufficiently large $N$ the projection onto the closure of $\operatorname{span}\{e_n : n \geq N\}$ is an infinite quantum anticlique for $\V$. Otherwise, if some $A_k$ has infinite rank, then the projection onto the closure of the range of $A_k$ is an infinite quantum obstruction for $\V$.
\end{proof}

\begin{lem} \label{lem:diag-infinite-dim-case-1}
Let $\V \subseteq \B(H)$ be a simultaneously diagonalizable operator system on an infinite dimensional Hilbert space and let $(A_n)$ be a linearly independent sequence of operators in $\V$. Then there is a sequence $(B_n)$ of operators in $\V$, an infinite rank projection $P \in \B(H)$ and an orthonormal basis $(f_n)$ for $PH$ such that the sequence of compressions $(P B_n P)$ is linearly independent and for $n \geq 1$, $\langle B_n f_i, f_i \rangle = 0$ for $1 \leq i < n$ and $\langle B_n f_n, f_n \rangle = 1$.
\end{lem}

\begin{proof}
Let $(h_n)$ be an orthonormal basis for $H$ that simultaneously diagonalizes $(A_n)$. We inductively construct sequences $(B_k)$ and $(f_k)$ and a sequence $(n_k)$ in $\bN$ such that for $k \geq 1$, $B_k \in \operatorname{span}\{A_1,\ldots,A_k\}$, $f_k = h_{n_k}$, $\langle B_k f_i, f_i \rangle = 0$ for $1 \leq i < k$ and $\langle B_k f_k, f_k \rangle = 1$.

For $k = 1$ choose $n_1 \geq 1$ such that $\langle A_1 h_{n_1}, h_{n_1} \rangle \ne 0$. Then we can take $B_1 = A_1 / \langle A_1 h_{n_1}, h_{n_1} \rangle$ and $f_1 = h_{n_1}$. Suppose $f_1,\ldots,f_{k-1}$ and $B_1,\ldots,B_{k-1}$ have been chosen. By the induction hypothesis and the linear independence of the sequence $(A_n)$ there is nonzero $A \in \operatorname{span}\{A_1,\ldots,A_k\}$ such that $\langle A h_{n_i}, h_{n_i} \rangle = \langle A f_i, f_i \rangle = 0$ for $1 \leq i < k$. Choose $n_k \in \bN$ such that $\langle A h_{n_k}, h_{n_k} \rangle \ne 0$. Then we can take $B_k = A / \langle A e_{n_k}, e_{n_k} \rangle$ and $f_k = h_{n_k}$. Proceeding inductively, we obtain the sequences $(B_n)$ and $(f_n)$ as desired. We can take $P$ to be the projection onto the closure of $\operatorname{span}\{f_n : n \geq 1\}$.
\end{proof}

\begin{lem} \label{lem:diag-infinite-dim-case-2}
Let $V \subseteq \B(H)$ be an operator system on a countably infinite dimensional Hilbert space. Let $(A_n)$ be a linearly independent sequence of simultaneously diagonalizable operators in $\V$ and let $(h_n)$ be an orthonormal basis that simultaneously diagonalizes $(A_n)$. Suppose that for each $n$, $\langle A_n h_i, h_i \rangle = 0$ for $1 \leq i < n$ and $\langle A_n h_n, h_n \rangle = 1$. Then $\V$ has an infinite quantum clique.
\end{lem}

\begin{proof}
Let $(e_n)$ denote the standard orthonormal basis for the Hilbert space $\ell^2(\bN)$. We identify the $m \times m$ matrices $M_m(\bC)$ with the set of operators $T \in \B(\ell^2(\bN))$ satisfying $\langle T e_i, e_j \rangle = 0$ for $\max\{i,j\} > m$. For $n \geq 1$ and $i \geq 1$, let $\alpha_{ii}^n = \langle A_n h_i, h_i \rangle$, so that
\[
A_n = \sum _{i \geq 1} \alpha_{ii}^n h_i h_i^*,
\]
where the sum is taken in the weak topology.

We will inductively construct a sequence $(x_n)$ of nonzero vectors in $\ell^2(\bN)$ satisfying very restrictive norm conditions. Also, for $n \geq 1$ and $m \geq 1$ satisfying $(m-1)^2 < n \leq m^2$, we will require that $x_n \in \operatorname{span}\{ e_1,\ldots,e_m \}$. Finally, we will require that the rank one operators $x_1 x_1^*,\ldots,x_n x_n^*$ are linearly independent and that the operators $A_1^n,\ldots,A_n^n$ are linearly independent, where for $1 \leq j \leq n$, $A_j^n$ is defined by
\[
A_j^n = \sum_{1 \leq i \leq n} \alpha_{ii}^j x_i x_i^* \in M_m(\bC).
\]

For $k = 1$, we can take $x_1 = 2^{-2} e_1$. For $k \geq 2$, suppose $x_1,\ldots,x_{k-1}$ have been chosen and let $m \geq 1$ satisfy $(m-1)^2 < k \leq m^2$. Since $M_m(\bC)$ is spanned by positive rank one matrices, there is a vector $x_k \in \operatorname{span}\{e_1,\ldots,e_m\}$ such that $x_k x_k^* \notin \operatorname{span}\{x_1 x_1^*,\ldots, x_{k-1} x_{k-1}^* \}$. Observe that for $1 \leq j \leq k$, $A_j^k = A_j^{k-1} + \alpha_{kk}^j x_k x_k^*$. In particular, $A_k^k = x_k x_k^*$.

By assumption, $A_1^{k-1},\ldots,A_{k-1}^{k-1}$ are linearly independent. Since small perturbations of finitely many linearly independent matrices are linearly independent, we can scale $x_k$ if necessary to ensure that $A_1^k,\ldots,A_{k-1}^k$ are linearly independent.

We claim that $A_1^k,\ldots,A_k^k$ are linearly independent. To see this, suppose there are scalars $\beta_1,\ldots,\beta_k \in \bC$ such that $\sum_{j=1}^k \beta_j A_j^k = 0$. Then 
\[
x_k x_k^* = - \sum_{j=1}^{k-1} \beta_j A_j^{k-1} - \sum_{j=1}^{k-1} \beta_j \alpha^j_{kk} x_k x_k^*.
\]
The linear independence of $x_1x_1^*, \ldots, x_k x_k^*$ implies $\sum_{j=1}^{k-1} \beta_j A_j^{k-1} = 0$ and $\sum_{j=1}^{k-1} \beta_j \alpha^j_{kk} = -1$, which contradicts the linear independence of $A_1^{k-1},\ldots,A_{k-1}^{k-1}$. 

By construction, for $m \geq 1$, $A_1^{m^2},\ldots,A_{m^2}^{m^2}$ span $M_m(\bC)$. For $1 \leq r,s \leq m$, let $E_{rs} = e_r e_s^*$. Then there are scalars $\gamma_1^{rs,m^2},\ldots,\gamma_{m^2}^{rs,m^2} \in \bC$ such that
\[
E_{rs} = \sum_{j=1}^{m^2} \gamma_j^{rs,m^2} A_j^{m^2}.  
\]
For $k > m^2$, we may further scale $x_k$ if necessary, while preserving the linear independence of $A_1^k,\ldots,A_k^k$, to ensure that
\[
\|x_k\|^2 \leq \frac{1}{2^{k+1} m \sum_{j=1}^{m^2} |\gamma_j^{rs,m^2}| \|A_j\|}.
\]
Proceeding inductively, we obtain the sequence $(x_n)$ as desired.

By Lemma \ref{lem:dilation} there is an isometry $V : \ell^2(\bN) \to H \oplus H$ such that
\[
V^*(A_k \oplus 0_H)V = \sum_{i \geq 1} \alpha_{ii}^k x_i x_i^*,
\]
where the sum is taken in the weak topology. For $m \geq 1$ and $1 \leq r,s \leq m$, we estimate
\begin{align*}
\left\| \vphantom{\sum_{j=1}^{m^2}} E_{rs} - \right. & \left. \sum_{j=1}^{m^2} \gamma_j^{rs,m^2} V^* A_j V \right\| \\
&= \left\| E_{rs} - \sum_{j=1}^{m^2} \gamma_j^{rs,m^2} A_j^{m^2} - \sum_{j=1}^{m^2} \gamma_j^{rs,m^2} \sum_{k>m^2} \alpha_{kk}^j x_k x_k^* \right\| \\
&= \left\| \sum_{j=1}^{m^2} \gamma_j^{rs,m^2} \sum_{k>m^2} \alpha_{kk}^j x_k x_k^* \right\| \\
&\leq \sum_{j=1}^{m^2} |\gamma_j^{rs,m^2}| \|A_j\|  \sum_{k>m^2} \|x_k\|^2 \\
&\leq \frac{1}{m}.
\end{align*}
It follows that $E_{rs}$ belongs to the norm closure of $V^* \V V$. Hence the weak closure of $V^* \V V$ contains all of the finite rank operators in $\B(\ell^2(\bN))$. In particular, it is weakly dense in $\B(\ell^2(\bN))$. We now conclude from Lemma \ref{lem:isometry} that $\V$ has an infinite quantum clique.
\end{proof}

\begin{prop} \label{prop:diagonalizable}
Let $H$ be a Hilbert space of countably infinite dimension and let $\V \subseteq \B(H)$ be a diagonalizable operator system. If the dimension of $\V$ is finite, then either it has an infinite quantum anticlique or it has an infinite quantum obstruction. Otherwise, if the dimension of $\V$ is infinite, then it has an infinite quantum clique.
\end{prop}

\begin{proof}
If the dimension of $\V$ is infinite, then applying Lemma \ref{lem:diag-infinite-dim-case-1} followed by Lemma \ref{lem:diag-infinite-dim-case-2}, we conclude that $\V$ has an infinite quantum clique. Otherwise, if the dimension of $\V$ is finite, then Lemma \ref{lem:diag-finite-dim-case} implies that either $\V$ has an infinite quantum anticlique or it has an infinite quantum obstruction.
\end{proof}

\section{Main results} \label{sec:main}

\begin{thm} \label{thm:main-end}
Let $\V \subseteq \B(H)$ be an operator system on an infinite dimensional Hilbert space $H$. Then it has either an infinite quantum clique, an infinite quantum anticlique or an infinite quantum obstruction.
\end{thm}

\begin{proof}
By compressing to a subspace, we can assume that $H$ is countably infinite dimensional. If there is an infinite rank projection $P$ such that the compression $P \V P$ is diagonalizable, then we can replace $\V$ with $P \V P$ and apply Proposition \ref{prop:diagonalizable}. Otherwise, by Proposition \ref{prop:pre-diagonalizable}, either $\V$ has an infinite quantum clique, or there is an infinite rank projection $P \in \B(H)$ and an operator subsystem $\V' \subseteq \V$ such that the compression $P \V' P$ is infinite dimensional and diagonalizable. In the latter case, applying Proposition \ref{prop:pre-diagonalizable} gives the desired result.
\end{proof}

The following corollary is an immediate consequence of Definition~\ref{defn:main}, Definition \ref{defn:almost-anticlique} and Theorem \ref{thm:main-end}.

\begin{cor}
Let $\V \subseteq \B(H)$ be an operator system on an infinite dimensional Hilbert space $H$. Then it has either an infinite quantum clique or an infinite quantum almost anticlique.
\end{cor}



\end{document}